\documentclass[12pt]{amsart}

\usepackage{bm,latexsym,amsfonts,amsmath,amssymb,amsthm,url,graphics,psfrag,amscd,stmaryrd, mathrsfs}
\usepackage[english]{babel}
\usepackage[T1]{fontenc}

\usepackage{graphics}
\usepackage{graphicx}
\usepackage{color}

\newcommand{\N}{\mathbb{N}}
\newcommand{\E}{\mathbb{E}}
\newcommand{\Z}{\mathbb{Z}}

\newcommand{\T}{\mathbb{T}}
\newcommand{\pp}{\mathbb{P}}
\newcommand{\kD}{\mathcal{D}}

\newcommand{\kO}{\mathcal{O}}

\newcommand{\kE}{\mathcal{E}}

\newcommand{\kI}{\mathcal{I}}
\newcommand{\kS}{\mathcal{S}}
\newcommand{\kN}{\mathcal{N}}

\newcommand{\pn}{\mathbb{P} }

\newcommand{\en}{\mathbb{E} }

\newcommand{\bin}{\textrm{Bin}}
\newcommand{\poi}{\textrm{Poi}}

\newcommand{\lin}{\left[\kern-0.15em\left[}
\newcommand{\rin} {\right]\kern-0.15em\right]}
\newcommand{\linf}{[\kern-0.15em [}
\newcommand{\rinf} {]\kern-0.15em ]}
\newcommand{\ilin}{\left]\kern-0.15em\left]}
\newcommand{\irin} {\right[\kern-0.15em\right[}

\newtheorem{lem}{Lemma}[section]

\newtheorem{prop}[lem]{Proposition}
\newtheorem{theo}[lem]{Theorem}

\newtheorem {rem}[lem] {Remark}

\title [Contact process on inhomogeneous random graphs]
{Exponential extinction time of the contact process on rank-one inhomogeneous random graphs
}

\author{ Van Hao Can 
}

\address{  Aix Marseille Universit\'e, CNRS, Centrale Marseille, I2M, UMR 7373, 13453 Marseille, France}
\address{ Institute of Mathematics, Vietnam Academy of Science and Technology, 18 Hoang Quoc Viet, 10307 Ha Noi, Viet Nam}
              \email{cvhao89@gmail.com}           

\begin{document}

\maketitle
\begin{abstract}
We show that the contact process on the rank-one inhomogeneous random graphs and Erdos-R\'enyi graphs with mean degree large enough  survives a time exponential in the size of these graphs for any positive infection rate. In addition, a metastable result for the extinction time is also proved.
\end{abstract}
\section{Introduction}
The contact process was introduced in  \cite{H} by T. E. Harris and is also often interpreted as a model for the spread of a virus in a population.  Given a locally finite graph $G=(V,E)$ and $\lambda >0$, the contact process on $G$ with infection rate $\lambda$ is a Markov process $(\xi_t)_{t\geq 0}$ on $\{0,1\}^V$. Vertices of $V$ (also called sites) are regarded as individuals which are either infected (state $1$) or healthy (state $0$). By considering $\xi_t$ as a subset of $V$ via $\xi_t \equiv \{v: \xi_t(v)=1\},$ the transition rates are given by
\begin{align*}
\xi_t \rightarrow \xi_t \setminus \{v\} & \textrm{ for $v \in \xi_t$ at rate $1,$ and } \\
\xi_t \rightarrow \xi_t \cup \{v\} & \textrm{ for $v \not \in \xi_t$ at rate }  \lambda \, |\{w \in \xi_t: \{v,w\} \in E\}|.
\end{align*}
 Given  $A \subset V$, we denote by $(\xi_t^A)_{t \geq 0}$ the contact process with initial configuration $A$ and if $A=\{v\}$ we simply write $(\xi^v_t)$. 

Since the contact process is monotone in $\lambda$,  we can define the critical value 
\begin{align*}
 \lambda_c(G)= \inf \{\lambda: \pp(\xi ^v_t \neq \varnothing \,\forall t) >0\}.
\end{align*} 
If $G$ is connected, this definition does not depend on the choice of $v$. For integer lattices,  it has been proved that $\lambda_c(\Z^d) $ is positive.  Interestingly,  the contact process on finite boxes $\llbracket 0,n \rrbracket^d$  exhibits a phase transition at the same critical value. More precisely, if we define the extinction time of the process $(\xi_t^{\bf 1})$ starting from full occupancy by
\begin{align*}
\tau_n = \inf \{t \geq 0: \xi^{\bf 1}_t= \varnothing\},
\end{align*}
then with high probability (w.h.p.) $\tau_n$ is of logarithmic order   when 
$\lambda<\lambda_c(\Z^d)$, and of exponential order when 
$\lambda>\lambda_c(\Z^d)$, see e.g. \cite{L} Section I.3.

\vspace{0.2 cm}
Recently, this phenomena has been observed in a few of other situations:  the random regular graphs  and  their limit  in the sense of the Benjamini--Schramm's local weak convergence of graphs \cite{BS} - the homogeneous tree, see  \cite{LS,MV};    the  finite homogeneous trees  and their limit - the canopy tree,  see  \cite{CMMV,MV,S};  the configuration models with heavy tail degree distributions and  their limit - the Galton-Watson tree, see  \cite{CD,CS,MMVY,MVY,P}; the preferential attachment graphs and their limit - the P\'olya-point random graph, see \cite{BBCS,C1}; random geometric graphs, see \cite{C2,MS}. 

\vspace{0.2 cm}
The goal of this paper is to explore the case of rank-one inhomogeneous graphs. An inhomogeneous random graph (IRG), $G_n=(V_n,E_n)$, is  defined as follows. Let $V_n= \{v_1,\ldots, v_n\}$ be the vertex set and let $(w_i)$ be a sequence of i.i.d. positive random variables with the same law as $w$. Then for any $1\leq i \neq j \leq n$, we independently draw an edge between $v_i$ and $v_j$ with probability 
$$p_{i,j}= 1 - \exp(- w_i w_j/ \ell_n),$$
where 
$$\ell_n = \sum_{i=1}^n w_i.$$
It is shown in  \cite[Vol. 2, Chapter 3]{H15} that when $\en(w)$ is finite,  $G_n$  converges weakly to a two-stages Galton-Watson tree. In this  tree,   the reproduction law of the root is  $(p_k)$ and the one of other vertices is $(g_k)$   with
\begin{align} \label{dpk}
 p_k = \pp(\poi(w)=k)= \en \left( e^{-w} \frac{w^k}{k!} \right) 
\end{align} 
and
\begin{align} \label{dgk}
g_k = \pp(\poi(w^*)=k)= \frac{1}{\en(w)}\en \left( e^{-w} \frac{w^{k+1}}{k!} \right),
\end{align}
where $w^*$ is the size-bias distribution of $w$.  We also assume in addition that 
\begin{itemize}
\item[$(H1)$] $w \geq 1$ a.s. and $\en(w) < \infty$, \\
\item[$(H2)$] the limiting tree is super critical, or equivalently
\begin{align*}
\nu=: \en(g)= \frac{\en(w^2)}{\en(w)}   >1,
\end{align*}
\item[$(H3)$] there exits a function $\varphi(k)$ increasing to infinity, such that
\begin{align*}
\limsup_{k \rightarrow \infty} g_k e^{k/ \varphi(k)} \geq 1.
\end{align*}
\end{itemize}
\begin{theo}\label{mt}
Let $\tau_n$ be the extinction time of the contact process on inhomogeneous random graphs with the weight $w$ satisfying the hypotheses $(H1)-(H3)$, starting from full occupancy. Then for any $\lambda>0$, there exist  positive constants $c$ and $C$, such that 
\begin{align*}
\pp(\exp (Cn) \geq \tau_n \geq \exp(cn)) \rightarrow 1 \quad \textrm{as } n\rightarrow \infty.
\end{align*} 
Moreover,
\begin{align*} 
\frac{\tau_n}{\en (\tau_n)}\quad  \mathop{\longrightarrow}^{(\kD)}_{n\to \infty} \quad  \kE(1),
\end{align*}
with $\kE(1)$ the exponential random variable with mean $1$.
\end{theo}
For simplicity, we will replace the hypothesis $(H3)$ by the following stronger version: there exits a  function $\varphi(k)$ increasing to infinity, such that 
\begin{equation}
\tag{$H3'$}  g_k e^{k/ \varphi(k)} \geq 1 \quad \textrm{ for all } k \geq 1.
\end{equation}
We will see in Section 3 that this assumption does not change anything to the proof.

\vspace{0.2 cm}
 We note that $(H1)$ is necessary for the weak convergence,  $(H2)$ is essential since without it w.h.p. all  components have size  $o(n)$  and $(H3)$ is the key hypothesis in our proof.

 We also remark that  under $(H2)$ and  $(H3)$, Pemantle proves in \cite{P} that $\lambda_c(GW(g))$ is zero, with  $GW(g)$ the Galton-Watson tree with  reproduction law $g$.  
 
 It is worth noting that studying $\lambda_c(GW(g))$ when $g_k \asymp \exp(-c k)$ is still a challenge. An equivalently interesting problem is to study the extinction time of the contact process on a super-critical Erdos-R\'enyi graph, which is a special inhomogeneous graph and converges weakly to a Galton-Watson tree with  Poisson reproduction law.

\vspace{0.2 cm}
Let us make some comments on the proof of Theorem \ref{mt}. The upper bound on $\tau_n$ follows from a general result in \cite{C2}. To prove the lower bound, we will show that $G_n$ contains a sequence of  disjoint star graphs  with large degree, whose total size  is of order $n$. Moreover, the distance between two consecutive star graphs is not too large, so that  the virus starting from a star graph can infect the other one  with high probability.  Then by comparing with an oriented percolation with density close to $1$, we get the lower bound. The convergence in law can be  proved similarly as in \cite{C2}.   

\vspace{0.2 cm}
 The paper is organized as follows. In Section 2, we prove some preliminary results to describe the neighborhood of a vertex in the graph. In Section 3, by defining some  exploration process of the vertices  we prove the existence of the sequence of star graphs mentioned above. Then we prove our main theorem. In Section 4, we prove a similar result for the Erdos-R\'enyi graph: for any $\lambda >0$, if the mean degree of the Erdos-R\'enyi graph is larger than some explicit function of $\lambda$, then  the extinction time is also of  exponential order. 

\vspace{0.2 cm}
 Now we introduce some notation. We denote the indicator function of a set $E$ by ${\bf 1}(E)$.  
For any vertices $v$ and $w$ we write $v \sim w$ if there is an edge between them.  We call size of a graph $G$ the cardinality of its set of vertices, and we denote it by $|G|$.  A graph in which all vertices have degree one, except one which is connected to all the others is called a 
\textit{star graph}. The only vertex with degree larger than one is called the center of the star graph, or central vertex. 

\vspace{0.2cm}
Furthermore we denote by $\bin(n,p)$ the binomial distribution with parameters $n$ and $p$ and denote by $\poi(\mu)$ the Poisson distribution with mean $\mu$. Let $X$ and $Y$ be two random variables or two distributions, we write $X \preceq Y$ if $X$ is stochastically dominated by $Y$.  If $f $ and $g$ are two real functions, we write $f= \mathcal{O}(g)$ if there exists a constant $C>0,$ such that $f(x) \leq C g(x)$ for all $x ;$  $f \asymp g $ if $f= \mathcal{O}(g)$ and $g= \mathcal{O}(f);$  $f=o(g)$ if $f(x)/g(x) \rightarrow 0$ as $x \rightarrow \infty$. 
Finally  for a sequence of random variables $(X_n)$ and a function $f:\N \to (0,\infty)$, we say that $X_n \asymp f(n)$ holds w.h.p. if there exist positive constants $c$ and $C,$ such that $\pn(c f(n) \leq X_n \leq C f(n)) \rightarrow 1$, as $n\to \infty$.

\section{Preliminaries}
\subsection{A preliminary result on the sequence of weights.}
\begin{lem} \label{lx}
 Let $(w_i)$ be the sequence of i.i.d. weights as in the definition of IRGs.  Then for any $\delta >0$, there exists $\varkappa_1 = \varkappa_1(\delta)  \in (0,1)$, such that 
\begin{align*}
\pp \left( \sum_{i \in U} w_i \geq (1-\delta)\sum _{i=1}^n w_i  \, \textrm{ for all } \, U \subset \{1, \ldots, n\} \textrm{ with } |U| \geq n(1-\varkappa_1) \right) \rightarrow 1.
\end{align*} 
\end{lem}
\begin{proof} Using the law of large numbers, we get
\begin{align*}
\pp\left(\sum_{i=1}^n w_i = n(\mu + o(1)) \right) \rightarrow 1,
\end{align*}
with
$$\mu = \E(w).$$
Observe that the worst case of the sum $\sum_{i \in U} w_i$ is where $U$ contains the first $[n(1- \varkappa) n]$ order statistics. Hence, it is sufficient to show that for any $\delta >0$, there exists $\varkappa_1 \in (0,1)$, such that
\begin{align} \label{cd}
\pp \left( \frac{1}{n}\sum_{i = n- [\varkappa_1 n]}^ n  w_{(i)}  \geq   \mu \delta  \right) \longrightarrow 0,
\end{align}
where $w_{(1)} \leq w_{(2)} \leq \ldots \leq w_{(n)}$ is the order statistics of the sequence $(w_i)$. Clearly, if $w$ is  bounded (and thus the sequence $(w_i)$ is also), i.e. there exists a positive constant $K$, such that $\pp(w \leq K) = 1$, then \eqref{cd} easily follows by taking $\varkappa_1 = (\mu\delta)/(2K)$. We assume that $w$ is unbounded. Since $\E(w) \in (0, \infty)$, we can find a number $x \in (0, \infty)$, such that  
\begin{align*}
 \E(w1(w>x)) \leq   (\mu \delta)/2.
\end{align*}
Let us define 
$$\Lambda = \{i: X_i >x \}.$$
Since $w$ is unbounded, 
\begin{equation*}
\varkappa_1:= \frac{\pp(w >x)}{2} \in (0,1).
\end{equation*}
By law of large number, we have 
\begin{align}
 \frac{|\Lambda|}{n}  &\mathop{\longrightarrow}^{\textrm{a.s.}} 2\varkappa_1, \label{bb1} \\
  \frac{1}{n} \sum _{ i \in \Lambda} w_i  = \frac{1}{n} \sum_{i=1}^n w_i 1 (w_i >x) \label{bb2}&\mathop{\longrightarrow}^{\textrm{a.s.}} \E(w1(w>x )) \in (0, \mu \delta/2).
\end{align}
We observe that 
\begin{equation}
\sum \limits_{i \in \Lambda} w_{i} = \sum _{i = n- |\Lambda| +1}^ n w_{(i)}. \label{bb3}
\end{equation}
Now, we can deduce \eqref{cd}   from \eqref{bb1}, \eqref{bb2} and \eqref{bb3}. 
\end{proof}
\subsection{Coupling of an IRG with a Galton-Watson tree}
We describe here the  neighborhood of a vertex in a set.  Let $U \subset V_n$ and $v \in V_n \setminus U$, and let $R$ be a positive integer. We denote by $B_R(v,U)$ the graph containing all vertices in $U$ at distance less than or equal to $R$ from $v$. We adapt the construction  in  \cite[Vol. 2, Section 3.4]{H15} to make a coupling between $B_R(v,U)$ and a {\it marked mixed-Poisson Galton-Watson tree}. 

\vspace{0.3 cm}
Conditionally on the weights $(w_i)$, we define a random variable $M_U$ as follows 
\begin{align} \label{ma}
\pp(M_U=m) = w_m/ \ell_U \quad  \textrm{ for all indices $m$ such that } v_m \in U, 
\end{align} 
with $\ell_U= \sum w_i 1(v_i \in U)$. Note that $\ell_n = \ell_{V_n}$.

\vspace{0.2 cm}
We define a random tree with  root $o$ as follows. We first define the mark of the root as $M_{o} = m$, with $m$ the index such that $v_m=v$. Then  $o$ has $X_{o}$ children, with
$$X_{o} \sim \poi(w_{M_{o}} \ell_U/ \ell_n).$$
Each child of the root, say $x$, is assigned an independent mark  $M_x$, with the same distribution as $M_U$. Conditionally on $M_x$, the number of children of $x$  has  distribution  $\poi(w_{M_x} \ell_U/\ell_n)$. 

\vspace{0.2 cm}
Suppose that all the vertices at height smaller than or equal to $i$ are defined. We  determine the vertices at height $(i+1)$ as follows. Each vertex at height $i$, say $y$, has an independent mark $M_{y}$ with the same distribution as $M_U$ and it has $X_{y}$ children, where $X_{y}$ is a Poisson random variable with mean $w_{M_{y}} \ell_U/ \ell_n$. 

\vspace{0,2 cm}
We denote the resulted tree by  $\T(v,U)$ and call it the {\it marked mixed-Poisson  Galton-Watson tree} associated  to $(v,U)$. In order to make a relation between  $B_R(v,U)$ and $\T(v,U)$, we define a {\it thinning }  procedure on $\T(v,U)$ as follows.

\vspace{0.2 cm}
 For a vertex $y $ different from  the root, we thin $y$ when either one of the  vertices on the unique path between the root  and $y$ has been thinned, or when $M_{y}= M_{y'}$, for some  vertex $y'$ on this path. 

\vspace{0.2 cm} 
 We denote by $\tilde{\T}(v,U)$ the tree resulting  from the thinning on $\T(v,U)$. 
\begin{prop} \label{ppe} \cite[Vol. II, Proposition 3.10]{H15} Conditionally on $(w_i)$, the set of vertices in $U$ at distance $k$ from $v$ (considering the graph induced in $U \cup \{v\}$) has the same distribution as 
\begin{align*}
\left( \{v_{M_{x}}: x \in \tilde{\T}(v,U) \textrm{ and } |x| =k\} \right)_{k\geq 0},
\end{align*} 
with  $|x|$ the height of $x$.  Moreover,   $B_R(v,U)$ contains a subgraph which has the same law as    $\tilde{\T}_R(v,U)$- the graph containing all vertices in $\tilde{\T}(v,U)$ whose heights are smaller than or equal to $R$.  
\end{prop} 
We note that in \cite{H15}, the author only proves this proposition for $U=V\setminus \{v\}$. The proof for any subset of $V\setminus \{v\}$  is essentially the same, so we do not present  here.

\vspace{0.2 cm}
\noindent {\bf The law of the marked-mixed Poisson Galton-Watson tree}. The offspring distribution of the root  is given by
\begin{align*}
p^U_k = \pp(\poi(w_{M_o} \ell_U/\ell_n)=k) \quad \textrm{ for } \quad k \geq 0.
\end{align*} 
The individuals of the second and further generations have the same offspring distribution, denoted  by $(g^U_k)$. It is given as follows: for all $k\geq 0$
\begin{align*}
g^U_k= \pp(\poi(w_{M_U} \ell_U/\ell_n)=k)= \sum\limits_{v_i \in U} \pp(\poi(w_i \ell_U/ \ell_n)=k) w_i/ \ell_U.
\end{align*} 
If $U=V_n$, we write $g^{(n)}_k$ for $g^U_k$. Hence
\begin{align*}
g^{(n)}_k = \sum\limits_{i=1 } ^{n} \pp(\poi(w_i) =k) w_i/ \ell_n=  \pp(\poi(W_n^*)=k),
\end{align*}
with $W_n^*$  the size-bias distribution of the empirical mean weight $W_n= (w_1+ \ldots +w_n)/n$. 

\noindent It is shown in \cite{H15} that since the $(w_i)$ are i.i.d. with the same law as $w$ and $\E(w)$ is finite,
\begin{align*}
W_n \quad \mathop{\longrightarrow}^{(\kD)}_{n\to\infty} \quad w \quad \textrm{ and } \quad W_n^* \quad \mathop{\longrightarrow}^{(\kD)}_{n\to\infty} \quad w^*,
\end{align*} 
with $w^*$   the size-bias distribution of $w$. Therefore,  we have the following convergence.
\begin{lem} \label{gnk} \cite[Vol. II, Lemma 3.12]{H15}. For all $k\geq 0$,
\begin{align*}
\lim \limits_{n \rightarrow \infty} g^{(n)}_k = g_k,
\end{align*}
with $(g_k)$ as in \eqref{dgk}.
\end{lem}
Using Lemmas \ref{lx} and \ref{gnk}, we will show that the distribution $(g^U_k)$ approximates $(g_k)$, provided $|U|$ is large enough. 
\begin{lem} \label{lgak} For any $\varepsilon >0$ and $K \in \mathbb{N}$, there exists a constant $ \varkappa_2=\varkappa_2(\varepsilon,K)  \in (0, \varkappa_1(1/2))$, such that 
\begin{align*}
\pp \left( g^U_k \geq (1- \varepsilon)g_k \,\, \textrm{ for all } 0 \leq k \leq K \textrm{ and } U \subset V_n \textrm{ with } |U| \geq (1-\varkappa_2) n \right) \rightarrow 1.
\end{align*}
\end{lem}
\begin{proof}
If $g_k=0$, then $g_k^U \geq (1-\varepsilon) g_k$. Assume that $k \leq K$ and  $g_k>0$. We have
\begin{align*}
|g_k^U-g_k| \leq |g_k^U -g_k^{(n)}|+|g_k^{(n)}-g_k|.
\end{align*}
Lemma  \ref{gnk} implies that for all $n$ large enough
\begin{align} \label{gknk}
|g_k^{(n)}-g_k| \leq \varepsilon g_k/2.
\end{align}
On the other hand, 
\begin{eqnarray*} 
|g_k^U - g_k^{(n)}| & = & \,  \vline \, \sum_{v_i \in U} \pp\left( \poi(w_i \ell_U/\ell_n)=k\right) \frac{w_i}{\ell_U} - \sum_{v_i \in V_n} \pp\left( \poi(w_i )=k\right) \frac{w_i}{\ell_n} \, \vline  \\
 &\leq & \sum_{v_i \in U} | \pp\left( \poi(w_i \ell_U/\ell_n)=k\right) \frac{w_i}{\ell_U} -  \pp\left( \poi(w_i )=k\right) \frac{w_i}{\ell_n} |  + \sum_{v_i \not \in U} \frac{w_i}{\ell_n}  \\
& \leq & \sum_{v_i \in U} |  \pp\left( \poi(w_i \ell_U/\ell_n)=k\right)  -  \pp\left( \poi(w_i )=k\right) | \frac{w_i}{\ell_U}  \\
& & +   \sum_{v_i \in U}   \pp\left( \poi(w_i )=k\right) \left( \frac{w_i}{\ell_U} - \frac{w_i}{\ell_n} \right)   + \sum_{v_i \not \in U} \frac{w_i}{\ell_n} \\
&=& S_1 + S_2 +S_3.
\end{eqnarray*}
Here, we have used that $|x_i y_i -a_i b_i| \leq |x_i-a_i|y_i + |y_i-b_i|a_i$ for all $x_i,y_i,a_i,b_i \geq 0$. 

We now define 
\begin{align*}
f_k(x)= \pp(\poi(x)=k) =  \frac{e^{-x}x^k}{k!}.
\end{align*}
By the mean value theorem, for any $x<y$ 
\begin{align*}
|f_k(x)-f_k(y)| \leq \max_{x \leq u  \leq y } |f_k'(u)| |x-y|.
\end{align*}
If $k=0$, then 
\begin{align*}
|f_0(x)-f_0(y)| \leq e^{-x} (y-x).
\end{align*}
If $k\geq 1$ then for $u >0$,
\begin{eqnarray*}
f_k'(u)& =& e^{-u} (k-u)\frac{u^{k-1}}{k!}, \\ 
 f_k''(u)& =  &e^{-u} (u^2-2ku +k(k-1)) \frac{u^{k-2}}{k!}.  
\end{eqnarray*}
Thus  $f_k'(u) \leq 0 \leq f_k''(u)$ for all $u \geq 2k$.  Hence, for  $2k \leq x \leq u$, 
 $$|f'_k(u)| \leq |f_k'(x)| \leq f_k(x ).  $$
 On the other hand, $|f'_k(u)| \leq 1$ for all $u \geq 0$. Therefore,
 \begin{align*}
 \max_{x \leq u \leq y} |f_k'(u)| \leq 1(x \leq 2k) + f_k(x). 
 \end{align*}
 In summary, for all $k$ and $0 \leq x \leq y$
 \begin{align} \label{fxy}
 |f_k(x)-f_k(y)| \leq (1(x \leq 2k)+ f_k(x))(y-x).
 \end{align}
 Applying \eqref{fxy}, we get that if $\ell_U \geq \ell_n/2$ then 
 \begin{eqnarray*}
 |f_k(w_i \ell_U/\ell_n)-f_k(w_i)| &\leq  & \left( 1(w_i \ell_U / \ell_n \leq 2k) + f_k(w_i \ell_U/\ell_n) \right) \frac{w_i (\ell_n - \ell_U)}{\ell_n} \\
 & \leq & (  1(w_i \leq 4k) + e^{-w_i/2} w_i^k/k!) \frac{w_i (\ell_n - \ell_U)}{\ell_n}. 
 \end{eqnarray*}
 Therefore,
 \begin{eqnarray*}
 S_1 & = & \sum _{v_i \in U} |f_k(w_i \ell_U/\ell_n)-f_k(w_i)| \frac{w_i}{\ell_U} \\
 & \leq & \sum _{i =1}^n \frac{w_i^2 (\ell_n - \ell_U)}{\ell_U \ell_n} 1(w_i \leq 4k) + \sum _{i =1}^n \frac{e^{-w_i/2} w_i^{k+2}}{k!} \frac{(\ell_n - \ell_U)}{\ell_U \ell_n} \\
& \leq & \frac{4k(\ell_n - \ell_U)}{\ell_U } + \left( \frac{1}{\ell_n} \sum _{i =1}^n \frac{e^{-w_i/2} w_i^{k+2}}{k!} \right) \left(\frac{\ell_n - \ell_U}{\ell_U } \right).
 \end{eqnarray*}
Observe that 
\begin{align*}
 \frac{1}{n} \sum _{i =1}^n \frac{e^{-w_i/2} w_i^{k+2}}{k!} \quad \rightarrow \quad\E(e^{-w/2}w^{k+2}/k!) < \infty. 
\end{align*}
On the other hand,  $\ell_n \asymp n$. Therefore, 
\begin{align*}
 \frac{1}{\ell_n} \sum _{i =1}^n \frac{e^{-w_i/2} w_i^{k+2}}{k!} = \kO(1).
\end{align*}
Hence
\begin{align*}
S_1 = \kO \left(\frac{k(\ell_n- \ell_U)}{\ell_U} \right).
\end{align*}
Moreover,
\begin{eqnarray*}
S_2 & = &\sum_{v_i \in U} \pp(\poi(w_i)=k) |\frac{w_i}{\ell_U}- \frac{w_i}{\ell_n}|  \\
& \leq &\sum_{i=1}^n \frac{w_i(\ell_n -\ell_U)}{\ell_U\ell_n}  = \frac{\ell_n -\ell_U}{\ell_U},
\end{eqnarray*}
and
\begin{align*}
S_3 = \sum_{v_i \not \in U} \frac{w_i}{\ell_n} \leq \frac{\ell_n -\ell_U}{\ell_U}.
\end{align*}
In conclusion, if $\ell_U \geq \ell_n/2$ and $\ell_n \asymp n$ then
\begin{align*}
|g_k^U-g_k^{(n)}| \leq S_1 +S_2 +S_3 = \kO \left(\frac{k(\ell_n- \ell_U)}{\ell_U} \right).
\end{align*}
Therefore, there exists $\delta_k = \delta_k(k, \varepsilon, g_k)>0$, such that if  $\ell_U \geq (1- \delta_k)\ell_n$ then
\begin{align} \label{gkan}
|g_k^U-g_k^{(n)}| \leq \varepsilon g_k/2.
\end{align}
Define $\delta = \min \{ \delta_k : k\leq K \textrm{ and } g_k >0\}$. Then we have $\delta >0$.  Now, by Lemma \ref{lx} there exists $\varkappa_2 = \varkappa_1(\delta) \wedge \varkappa_1(1/2) \in (0,1)$, such that  w.h.p. for all $U \subset V_n$ with $|U| \geq (1 - \varkappa_2)n$,
$$\ell_U \geq (1- \delta) \ell_n.$$
Thus by \eqref{gknk} and \eqref{gkan}, w.h.p. for all $U \subset V_n$ with $|U| \geq (1 - \varkappa_2)n$,
\begin{align*}
g_k^U \geq (1- \varepsilon) g_k  \,\, \textrm{ for all } k \leq K,
\end{align*}
which proves the result.
\end{proof}
For any $\varepsilon \in (0,1)$ and $K \in \mathbb{N}$, we define a distribution $(g^{\varepsilon,K}_.)$ as follow:
\begin{align*}
g^{\varepsilon,K}_k &= 0 \hspace{3cm} \textrm{ if } \quad k\geq K+1,\\
g^{\varepsilon,K}_k &= (1- \varepsilon) g_k \hspace{1.6 cm} \textrm{ if } \quad 1 \leq k\leq K,\\
g^{\varepsilon,K}_0 &= 1- (1- \varepsilon) \sum_{k=1}^K g_k.
 \end{align*}
The following result is  a direct consequence of Lemma \ref{lgak}. 
\begin{lem} \label{gea}
For any $\varepsilon >0$ and $K \in \mathbb{N}$,
\begin{align*}
\pp \left(  (g^{\varepsilon,K}_.) \preceq (g^U_.)  \, \textrm{ for all $U \subset V_n$ with } |U| \geq (1- \varkappa_2 )n\right) \rightarrow 1,
\end{align*}
with $\varkappa_2$ as in Lemma \ref{lgak}.
\end{lem} 
Observe that $(g^{\varepsilon,K}_.)$ stochastically increases (resp. decreases) in $K$ (resp. $\varepsilon$). Moreover, it converges to $(g_.)$ as $\varepsilon \rightarrow 0$ and $K \rightarrow \infty$. Therefore, by the hypothesis $(H_2)$, there are positive constants $\varepsilon_0$ and $K_0$, such that for all $\varepsilon \leq \varepsilon_0$ and $K \geq K_0$,
\begin{align} \label{nuek}
\nu_{\varepsilon,K}:=\sum_{k=0}^{\infty} k g^{(\varepsilon,K)}_k \geq \bar{\nu},
\end{align} 
where 
\begin{align*}
\bar{\nu}= \frac{1+ \nu}{2} \in (1, \nu).
\end{align*}
Define for $K\geq K_0$,
\begin{align*} 
 \kE(K)= \{(g^{\varepsilon_0,K}_.) \preceq (g^U_.) \textrm{ and } \ell_U \geq \ell_n/2 \, \textrm{ for all $U \subset V_n$ with } |U| \geq (1- \varkappa_2)n\},
\end{align*}
with $\varkappa_2=\varkappa_2(\varepsilon_0,K)$ as in Proposition \ref{gea}. Using this proposition and Lemma \ref{lx} with the fact that $\varkappa_2 \leq \varkappa_1(1/2)$,  we obtain 
\begin{align} \label{kek}
\lim_{n\rightarrow \infty}\pn(\kE(K)) \rightarrow 1.
\end{align} 
We call $T_{\varepsilon,K}$ the Galton-Watson tree with  reproduction law $(g^{\varepsilon,K}_.)$. Then \eqref{nuek} implies that $T_{\varepsilon,K}$ is super critical when $\varepsilon \leq \varepsilon_0$ and $K\geq K_0$. From now on, we set 
$$\psi_1(K)=[K/\sqrt{\varphi(3K)}],$$
with the function $\varphi$ as in the hypothesis $(H3')$.
 
We  prove here a key lemma saying that when $U$ is large enough, with positive probability  there exists a  vertex in $\tilde{\T}(v,U)$ at distance less than   $\psi_1(K)$ from the root having more than $3K$ children (which implies that there exists a vertex with degree larger than $3K$ in $B_{\psi_1(K)}(v,U)$ with positive probability).   
\begin{lem} \label{lqk}
There are positive constants $\theta_1$ and $K_1$, such that for all $K\geq K_1$ and $U \subset V_n$ with $|U| \geq (1-\varkappa_2)n$ and $n$ large enough,
\begin{align*} 
  \pn \left( \exists \, x \in \tilde{\T}(v,U):    |x| \leq   \psi_1(K),   \deg(x) \geq  3K+1 \mid \kE(K)\right) \geq \theta_1,
\end{align*}
with $\varkappa_2$ as in Lemma \ref{gea}.
\end{lem}
\begin{proof} If $w_v$ -the weight of $v$- is larger  than $10K$, then  $\deg(o)$ is larger than $3K$ with positive probability.  Indeed, $\deg(o)$ is a Poisson random variable with parameter $w_v \ell_U/\ell_n$. Moreover, on $\kE(K)$, we have $\ell_U \geq \ell_n/2$.  Therefore,
\begin{eqnarray*}
\pn(\deg(o) \geq 3K \mid \kE(K)) &=& \pn(\poi(w_v \ell_U/\ell_n) \geq 3K \mid \kE(K)) \\
  &\geq& \pn(\poi(5K) \geq 3K) >0. 
\end{eqnarray*}
Hence, the result follows.  We now suppose that $w_v \leq  10K$.  Then,  in the proof of \cite[Vol. II, Corollary 3.13]{H15},  it is shown that for any $\ell$ 
\begin{align} \label{zvu}
\pn(\T_{\ell}(v,U) \equiv \tilde{\T}_{\ell}(v,U)) \rightarrow 1 \quad \textrm{as} \quad n\rightarrow \infty.
\end{align}
We denote by  
\begin{eqnarray*}
Z^{v,U}_{\ell} &=&  |\{x \in \T(v,U), |x| = \ell\}|. 
\end{eqnarray*}
Then by Lemma \ref{gea} conditionally on $\deg(o) \geq 1$,
\begin{align} \label{zdz}
Z^{v,U}_{\ell+1}  \succeq Z^{\varepsilon_0,K}_{\ell},
\end{align}
with $Z^{\varepsilon_0,K}_{\ell}$ the number of individuals at the  $\ell^{th}$ generation of $T_{\varepsilon_0,K}$.  We remark that
\begin{align*}
\mathbb{P}(Z_{\ell} \geq m^{\ell/2} \mid Z_{\ell} \geq 1 ) \rightarrow 1 \quad \textrm{ as } \ell \rightarrow \infty,
\end{align*}
where $Z_{\ell}$ is the number of individuals at the  $\ell^{th}$ generation of a Galton-Watson tree $T$ with mean $m>1$.  Therefore, for all $\ell$ large enough
\begin{align*} 
\mathbb{P}(Z_{\ell} \geq m^{\ell/2}) \geq \mathbb{P}(|T|=\infty)/2.
\end{align*}
Hence, for $K \geq K_0$
\begin{align} \label{zek}
\mathbb{P}(Z^{\varepsilon_0,K}_{\ell} \geq \bar{\nu}^{\ell/2}) \geq  \mathbb{P}(Z^{\varepsilon_0,K_0}_{\ell} \geq \bar{\nu}^{\ell/2})  \geq \mathbb{P}(|T_{\varepsilon_0,K_0}| = \infty)/2.
\end{align}
It follows from \eqref{zdz} and \eqref{zek} that  for all $\ell$ large enough,
 \begin{eqnarray} \label{zvul}
 \pn(Z^{v,U}_{\ell+1} \geq \bar{\nu}^{\ell/2} \mid \kE(K)) & \geq & \pn(\deg(o) \geq 1 \mid \kE(K))  \times \mathbb{P}(|T_{\varepsilon_0,K_0}| = \infty)/2 \notag \\
 &\geq & (1- e^{-w_v/2}) \times \mathbb{P}(|T_{\varepsilon_0,K_0}| = \infty)/2  \notag \\
 &\geq & \mathbb{P}(|T_{\varepsilon_0,K_0}| = \infty)/8.
 \end{eqnarray}
 Here, we have used that $w_v \geq 1$. 
Now, we set $$\ell = \psi_1(K)-1.$$
Then we have 
\begin{eqnarray} \label{tkr}
& &\pn \left(\exists \, x \in \T(v,U):  |x| = \ell+1, \deg(x) \geq 3K+1 \mid Z^{v,U}_{\ell+1} \geq \bar{\nu}^{\ell/2} \right) \\
 & \geq &  \pn\left(\bin\left([\bar{\nu}^{\ell/2}], g_{3K} \right) \geq 1 \right) \rightarrow 1 \qquad \textrm{ as } \qquad K \rightarrow \infty,  \notag
\end{eqnarray}
since under $(H3')$,
\begin{align*}
\bar{\nu}^{\ell/2} g_{3K} \geq \bar{\nu}^{\ell/2}  \exp(-3K/ \varphi(3K)) \rightarrow \infty \quad \textrm{ as } \quad K \rightarrow \infty.
\end{align*}
Now, the result follows from \eqref{zvu}, \eqref{zvul} and \eqref{tkr}.
\end{proof}
\section{Proof of Theorem \ref{mt}}
\subsection{Structure of the proof}
For $\ell, M \in \mathbb{N}$, we define the class $\kS(\ell,M)$ as the set of all graphs containing  a sequence of $\ell$ disjoint star graphs of size $M$ with centers $(x_i)_{i \leq \ell}$, such that $d(x_i, x_{i+1}) \leq \psi_1(M) +1 $ for all $i \leq \ell-1$.

The proof of Theorem \ref{mt} relies on the following propositions.
\begin{prop} \label{pcm}
For any positive integer $M$, there exist positive constants $c$ and $K$, such that $K \geq M$ and  w.h.p. $G_n$ belongs to the class $\kS([cn],K)$.
 \end{prop}
\begin{prop} \label{tng} Let $\tau_{\ell,M}$ be the extinction time of the contact process on a graph of the class $\kS(\ell,M)$ starting from full occupancy. Then there exist positive constants $c$ and $C$ independent of $\lambda$, such that if $ h(\lambda)M \geq C \psi_1(M) $, then
\begin{align} \label{taun} 
\pp(\tau_{\ell,M} \geq \exp(c \lambda^2  \ell M )) \rightarrow 1 \quad \textrm{as} \quad \ell \rightarrow \infty,
\end{align}
with  $h(\lambda)= \bar{\lambda}^2/|\log \bar{\lambda}|$ and $\bar{\lambda}= \lambda \wedge 1/2$.
\end{prop}
\begin{prop} \label{cvel}
Let $(G_n^0)$ be a sequence of connected graphs, such that $|G_n^0|\le n$ and $G_n^0$ belongs to the class $\kS(k_n,M)$, for some sequence $(k_n)$. 
Let $\tau_n$ denote the extinction time of the contact process on $G_n^0$ 
starting from full occupancy.  Then there exists a positive constant $C$, such that if $ h(\lambda) M \geq C \psi_1(M)$ with $h(\lambda)$ as in Proposition \ref{tng} and
\begin{align} \label{nas}
\frac{k_n}{d_n \vee \log n} \rightarrow \infty,
\end{align}
with  $d_n$ the diameter of $G_n^0$, then 
\begin{align*} 
\frac{\tau_n}{\en (\tau_n)}\quad  \mathop{\longrightarrow}^{(\kD)}_{n\to \infty} \quad  \kE(1).
\end{align*}
\end{prop}   
{\it Proof of Theorem \ref{mt}}. Observe that  Proposition \ref{pcm} and Lemma \ref{tng} imply the lower bound on $\tau_n$. On the other hand, the upper bound follows from  Lemma 3.4 in \cite{C2} and the fact that $|E_n| \asymp n$ w.h.p., see \cite[Vol. I, Theorem 6.6]{H15}. Finally, the convergence in law of $\tau_n/\en(\tau_n)$ can be proved similarly as in the proof of Theorem 1.1 in \cite{C2} by using Propositions \ref{pcm},  \ref{cvel}  and the following:
\begin{itemize}
\item[$\bullet$] w.h.p. $d_n = \kO(\log n)$ with $d_n$ the diameter of the largest component of the IRG, \\
\item[$\bullet$] w.h.p. the size of the second largest component in the IRG is $\kO(\log n)$.
\end{itemize}  
These claims are proved in Theorems 3.12 and 3.16 in \cite{BJR} for a general model of IRG.  \hfill $\square$

\vspace{0.2 cm}
{\it Proof of Proposition \ref{tng}}. Similarly to Lemma 3.2 in \cite{C2}, we can prove \eqref{taun} by using a comparison between the contact process  and an oriented percolation on $\llbracket 1,n \rrbracket$ with density close to $1$. Note that  here, we use a mechanism of  infection  between star graphs instead of complete graphs as it was the case in \cite{C2}. The mechanism for star graphs is described in Lemmas 3.1 and 3.2 in \cite{MVY} and   the function $h(\lambda)$ is chosen appropriately to apply these results. \hfill $\square$

\vspace{0.2 cm}
{\it Proof of Proposition \ref{cvel}}. The proof is the same as for Lemma 3.3 in \cite{C2}.  \hfill $\square$
\subsection{Proof of Proposition \ref{pcm}} This subsection is divided into four parts. In the first part, we define a preliminary process, called  an exploration, which uses Proposition   \ref{ppe} and Lemma \ref{lqk} to discover the neighborhood of a vertex. In Parts two and three, we describe the two main tasks, and the last part gives the conclusion.
\subsubsection{Exploration process} For a given vertex $v$ and a set $U \subset V_n$, we denote by $\kN(v,U)$ the set of neighbors of $v$ in $U$, i.e.
$$\kN(v,U)=\{w \in U : w \sim v\}.$$
Moreover, we will define  an {\bf exploration}  of $v$ in $U$ of type $K$ (and denote it by $E_K(v,U)$ and call $U$ the {\it source set} of the exploration). The aim of this exploration is  to find a vertex $x$ in $U$ with  degree larger than $3K$ at  distance less than $\psi_1(K)$ from $v$,  by exploring at most $\psi_2(K)$ vertices in $U$, with 
$$\psi_2(K)=(3K)^{\psi_1(K)+1}.$$

\vspace{0.15 cm}
First, we set $x_0=v$, $U_0=U$ and $W_0=\{x_0\}$ and call it   the {\it waiting set}. We define a sequence of trees $(T^{k}(v,U))_{k\geq 0}$ as the record of the exploration, starting with $T^{0}(v,U)=\{x_0\}$.  

\vspace{0.15 cm}
\begin{itemize}
\item[$\bullet$] If $\kN(x_0,U_0)= \varnothing$,  we  define  $U_1=U_0$ and $W_1=W_0 \setminus \{x_0\}$ and $T^1(v,U)=T^0(v,U)$. 
\item[$\bullet$] If $|\kN(x_0,U_0)| \geq 3K$,  we arbitrarily choose $3K$ vertices in $\kN(v,U)$ to form three {\it seed} sets of size $ K$  denoted by $F_{v,1}$,  $F_{v,2}$ and $F_{v,3}$. Then we declare that $E_K(v,U)$ is {\it successful}; we stop the exploration and define  
\begin{align*}
U_1=U \setminus (F_{v,1} \cup F_{v,2} \cup F_{v,3}).
\end{align*}
\item[$\bullet$] If $1 \leq |\kN(x_0,U_0)| < 3K$, we  define  
\begin{eqnarray*}
U_1 &= & U_0 \setminus \kN(x_0,U_0), \\
T^{1}(v,U)&=& T^{0}(v,U) \cup \kN(x_0,U_0) \textrm{ together with the edges} \\
 && \hspace{5cm} \textrm{ between $x_0$ and } \kN(x_0,U_0), \\ 
W_1 & = & (W_0 \setminus \{x_0\}) \cup \{x \in \kN(x_0,U_0): d_{T^1(v,U)} (x_0, x) \leq  \psi_1(K)\} , 
\end{eqnarray*}
with  $d_T(x,y)$  the graph distance between $x$ and $y$ in a tree $T$. 

\noindent Then we chose an arbitrary vertex $x_1$ in $W_1$ and repeat this step with $x_1$ and $U_1$ in place of $x_0$ and $U_0$. 
\end{itemize} 

\vspace{0.15 cm}
\noindent We continue like this until the waiting set is empty, or  we succeed at some step. 

\vspace{0.15 cm}
Note that after the $k^{th}$ step, we define 
\begin{displaymath}
T^{k+1}(v,U) = \left \{ \begin{array}{ll}
T^{k}(v,U) \cup \kN(x_k,U_k) \textrm{ together with the edges } & \\ 
\hspace{3.5 cm} \textrm{ between $x_k$ and } \kN(x_k,U_k)  & \textrm{ if } |\kN(x_k,U_{k})| < 3K \\
T^{k}(v,U) & \textrm{ if } |\kN(x_k,U_{k})| \geq 3K, 
\end{array} \right.
\end{displaymath}
and
\begin{displaymath}
U_{k+1} = \left \{ \begin{array}{ll}
U_{k} \setminus \kN(x_k,U_k) & \textrm{ if } |\kN(x_k,U_{k})| < 3K \\
U_{k} \setminus (F_{x_k,1} \cup F_{x_k,2} \cup F_{x_k,3} ) & \textrm{ if } |\kN(x_k,U_{k})| \geq 3K, 
\end{array} \right.
\end{displaymath}
and 
$$W_{k+1}= (W_k \setminus \{x_k\}) \cup \{x \in \kN(x_k,U_k): d_{T^{k+1}(v,U)} (x_0, x) \leq  \psi_1(K)\}. $$

\vspace{0.15 cm}
\noindent If the process stops after $k_0$ steps, we define the remaining source set
$$\tilde{U}=U_{k_0+1}.$$

\vspace{0.15 cm}
\noindent When an exploration is successful,  its outputs are  the set $\tilde{U}$ and  a vertex, say $u$, with three seed sets $F_{u,1}, F_{u,2}$ and $F_{u,3}$ of size $K$. Otherwise, the output is simply   $\tilde{U}$.

\begin{lem} \label{lep}
The following statements hold.
\begin{itemize}
\item[(i)] For all $v$ and $U$,  
\begin{align*}
|\tilde{U}| \geq |U| - \psi_2 (K),
\end{align*}
with $\psi_2(K)= (3K)^ {\psi_1(K)+1}$. 
\item[(ii)] For all $K \geq K_1$, we have
\begin{align*}
\lim_{n \rightarrow \infty} \pp(E_K(v,U) \textrm{ is successful} \mid \kE(K),  |U| \geq (1-  \varkappa_2)n)  \geq \theta_1,
\end{align*}
with $\varkappa_2$, $\theta_1$ and $K_1$ as in Lemmas \ref{gea} and \ref{lqk}. 
\end{itemize}
\end{lem}
\begin{proof}
Part (i) follows from the facts that  at each step we remove from the source set at most $3K$ vertices, and that we only explore the vertices at distance less than or equal to $\psi_1(K)$ from $v$.

\noindent  We now prove  (ii). Similarly to Lemma \ref{lqk}, if $w_v$ -the weight of $v$- is larger than $10K$, 
\[\pp(\deg(v)\geq 3K) \geq \theta_1,\]
and thus (ii) follows. Suppose that $w_v \leq 10K$. Then using the same argument in  Lemma \ref{lqk}, or \cite[Vol. II, Corollary 3.13]{H15} (for showing that $\T_{\ell}(v,U) \equiv \tilde{\T}_{\ell}(v,U) $ w.h.p. when $w_v$ is bounded), we get 
\begin{equation} \label{nbsk}
\pp(B_{\psi_1(K)}(v,U) \textrm{ has no cycle}) \rightarrow 1 \quad \textrm{as} \quad n\rightarrow \infty.
\end{equation}  
Suppose that $B_{\psi_1(k)}(v,U)$ is a tree (or has no cycle). Then the order of explorations of vertices in the waiting set  does not effect to the outcome of $E_K(v,U)$.  Thus 
\begin{equation*}
\{E_K(v,U) \textrm{ is successful}\} \supset \{\exists \, x \in B_{\psi_1(K)}(v,U): \deg(x) \geq 3K+1\} \cap \{B_{\psi_1(K)}(v,U) \textrm{ is a tree}\}.
\end{equation*}
 Therefore, the result follows from  Proposition \ref{ppe}, Lemma \ref{lqk} and \eqref{nbsk}. 
 \end{proof}
\subsubsection{Task I} The goal of this task is to show that w.h.p. by discovering $o(\log n)$ vertices, we can find in $G_n$ a subgraph belonging to  the class $\kS(L_n, 3K)$ with 
$$L_n = [\log \log \log n].$$
The proof of this task contains two parts. First, by introducing a process called {\it trail}, we show that given a large source set $U$, with positive probability by exploring at mots $o(\log n)$ vertices in $U$, we can find a graph of class  $\kS(L_n, 3K)$. Secondly, we continuously operate trails. By the first part, the  time to the first success is stochastically dominated by  a geometric random variable with positive mean, thus it is finite with high probability. Let us start with the definition of trail.   

\vspace{0.15 cm}
 For  $K \geq K_1$ and $U \subset V_n \setminus\{v\}$, with $|U| \geq (1- \varkappa_2/2)n$ and $\varkappa_2 = \varkappa_2(\varepsilon_0,K)$ as in Lemma \ref{gea}, we define a {\bf trial} $Tr(U,L_n,K)$ as follows. First, we take arbitrarily a vertex $v \in U$.
 
 \vspace{0.4 cm}
 At {\it level $0$}, we define $\tilde{W}_0= \{v\}$ and call it the waiting set at level $0$. Then we perform $E_K(v,U)$ and call $\tilde{U}_v$ the source set after this exploration. If it fails, we declare that $Tr(U,L_n,K)$ {\it fails}. Otherwise, we are now in {\it level $1$} and continue as follows.

\vspace{0.15 cm}
  Let $x_1$ be the vertex with degree larger than $3K+1$ which makes $E_K(v,U)$ successful and  let $F_{x_1,1}$,  $F_{x_1,2} $ and $F_{x_1,3}$ be its three seed sets of size $ K$.  We denote  $\tilde{W}_1=F_{x_1,3}$ and call it the waiting set at level 1 ($F_{x_1,1}$ and  $F_{x_1,2} $ are reserved for Task II). We sequentially perform explorations of vertices in $\tilde{W}_1$.  More precisely, we choose arbitrarily a vertex, say $y_1$ in $\tilde{W}_1$ and operate $E_K(y_1, \tilde{U}_v)$ and get a new source set $\tilde{U}_{y_1}$. Then we  operate $E_K(y_2,\tilde{U}_{y_1})$ with $y_2$ chosen arbitrarily from $\tilde{W}_1 \setminus \{y_1\}$, and so on.

\vspace{0.15 cm} 
If none of these explorations is successful, we  declare that the trial {\it fails}. 

\vspace{0.15 cm}
If some of those are successful, we are in {\it level $2$} and get some triples of  seed sets $F_{.,1}, F_{.,2}$ and $F_{.,3}$. Denote by $\tilde{W}_2$ the waiting  set at level $2$, which is the union of all the seed sets of the third type $F_{.,3}$.  Then we sequentially perform the explorations of vertices in  $ \tilde{W}_2$.

\vspace{0.15 cm}
We continue this process until  either we explore all vertices in waiting sets, or  when we exceed to the $L_n$-th level. 

\vspace{0.15 cm}
We declare that $Tr(U,L_n,K)$ is {\it successful} if we can access to the $L_n$-th level, or that it has failed  otherwise.

\vspace{0.15 cm} The output of a successful trial is a sequence of $L_n$ vertices associated with $L_n$ pairs of disjoint seed sets and a remained source set $\bar{U}$, with $|\bar{U}| \geq |U| - o(\log n)$. The output of a failed trail is simply $\bar{U}$. 

\begin{lem} \label{ltr} The following statements hold.
\begin{itemize}
\item[(i)] A trial discovers at most $K^{L_n+1} \psi_2(K)$ vertices.
\item[(ii)] There exist constants $\theta_2 \in (0,1)$ and $K_2 \geq K_1$, such that for any $K\geq K_2$ and $U \subset V_n $  satisfying  $|U| \geq (1- \varkappa_2/2)n$, we have
\begin{align*}
\lim_{n\rightarrow \infty} \pp(Tr(U,L_n,K)\textrm{ is successful } \mid \kE(K) ) \geq \theta_2, 
\end{align*}
with $\varkappa_2 $ and $ K_1$  as in Lemmas \ref{gea} and \ref{lqk}.
\end{itemize}
\end{lem}
\begin{proof} For (i), we observe that an exploration creates  at most one third type seed set  of size $K$. Then in a trial, we operate at most 
$$1+ K+ K^2 + \ldots+ K^{L_n} \leq K^{L_n+1} \quad \textrm{explorations.}$$ 
  Moreover, an exploration uses at most $\psi_2(K)$ vertices. Therefore, a trial discovers at most 
 $$ K^{L_n+1} \psi_2(K) = o(\log n) \textrm{ vertices}.$$
We now prove (ii). As a trial uses at most $o(\log n)$ vertices, and initially  $|U| \geq (1- \varkappa_1 /2)n$, during the trial $Tr(U,L_n,K)$ the source sets of explorations always have  cardinality larger than  $(1- \varkappa_2)n$.

Hence, by Lemma \ref{lep} (ii), for all $n$ and $K$ large enough,  on $\kE(K)$ each exploration in $Tr(U,L_n,K)$ is successful with probability larger than $\theta_1/2$. 

On the other hand, each successful exploration creates  $K$ new vertices in the next level.  Hence $(|\tilde{W}_i|)_{i \leq L_n}$- the numbers of vertices to explore up to the $L_n$-th level in the trial  stochastically dominate a branching process $(\eta_i)_{i\leq L_n}$ starting from  $\eta_0= 1 $  with reproduction law $\eta$ given by 
\begin{align*}
\pp(\eta = K) &= \theta_1/2 \\
\pp(\eta =0)& =1 -\theta_1/2.
\end{align*}
We choose $K$ large enough, such that $K \theta_1 >2$. Then  $(\eta_i)$ is  super critical, and thus
\begin{align*}
\pp(Tr(x,W,L_n,K) \textrm{ is successful } \mid \kE(K)) \geq \pp(\eta_{L_n} \geq 1) >0,
\end{align*}
which proves the result.  
\end{proof}
We  can now define  {\bf Task I} as follows. Let $U_1 = V_n$ and  operate $Tr(U_1, L_n,K)$. 
\begin{itemize}
\item[$\bullet$] If $Tr(U_1,L_n,K)$ is successful, we declare that Task I  is {\it successful}.
\item[$\bullet$] Otherwise, we call $\bar{U}_1$ the source set after this trial.  We then perform $Tr(U_2,L_n,K)$ with $U_2=\bar{U}_1$. 
\end{itemize}
We continuously operate   trials. We stop the task I if we get a success or we operate more than $L_n$ trails. 

 We declare that Task I  is {\it successful} if there is a successful trial and that it {\it fails} otherwise.

\begin{prop} \label{pc1} The following statements hold.
\begin{itemize}
\item[(i)] Task I uses at most $o(\log n)$ vertices. \\
\item[(ii)] We have
$$ \lim \limits_{n \rightarrow \infty} \pp(\textrm{Task I is successful} \mid \kE(K)) =1. $$
\end{itemize}
\end{prop}
\begin{proof}
By Lemma \ref{ltr} (i) in this task, we discover  at most 
$$L_n K^{L_n+1} \psi_2(K)=o(\log n) \textrm{ vertices}.$$
Hence, the cardinality of the source set  is always  larger than $n-o(\log n)$. Therefore, by Lemma \ref{ltr} (ii),  each trial is successful with probability larger than $\theta_2/2$ for $n$ large enough. 

We define $\kI$ the first index of the first successful trail (if there is no such index, we let $\kI= \infty$). Then conditionally on $\kI \leq L_n$, it is stochastically dominated by  a geometric random variable with parameter $\theta_2/2$. Therefore,  
\begin{align*}
\pp(\kI = \infty ) \leq (1- \theta_2/2)^{L_n} =o(1).
\end{align*}
In  other words, Task I is successful w.h.p.
 \end{proof}
\subsubsection{Task II} Suppose that Task I is successful.  Then there is a sequence of vertices $\{u_1, \ldots, u_{L_n}\}$, such that $d(u_{i-1},u_{i}) \leq \psi_1(K)$ for all $2\leq i \leq L_n$, together with  $L_n$ pairs of disjoint seed sets $(F_{u_1,1}, F_{u_1,2}), \ldots, (F_{u_{L_n,1}},F_{u_{L_n},2})$ attached respectively to $(u_i)$. Moreover, the cardinality of  $U^*$-the source set after Task I-  is larger than $  n -o(\log n)$.  Importantly, we have not yet discovered the vertices in  
$$U^+=U^* \cup \bigcup_{1 \leq i \leq L_n} (F_{u_i,1} \cup F_{u_i,2})   .$$
We will show that from $(u_i)$ and $U^+$, we can find a sequence of $\kO(n)$ good vertices which are associated with seed sets. This particularly shows the existence of a graph of class $\kS(\kO(n), K)$. To achieve this goal, the strategy is as follows. First, by introducing {\it experiments}, we show that with probability larger than $2/3$ starting from a good vertex, we can find another good vertex at distance less that $\psi_2(K)$. Moreover, each experiment only explores at most $\psi_3(K)$ vertices with 
$$\psi_3(K)=K \psi_2(K).$$
 Secondly, by the first step the number of good vertices stochastically dominates a random walk biased to the right. By a standard argument for random walk, we deduce that with high probability there are $\kO(n)$ good vertices.

\vspace{0.15 cm}
For any set $F \subset U^+$ of size $K $ and $S \subset U^* \setminus F$  we define an {\bf experiment} $Ep(F,S,K)$ as follows.

\vspace{0,4 cm}
 We write $F$ as $\{z_1, \ldots, z_K \}$. Then we  sequentially operate explorations $E_K(z_1, S_{z_0})$,  $\ldots $, $E_K(z_{K},S_{z_{K-1}})$, where  $S_{z_i}$ is the source set  after the exploration $E_K(z_{i},S_{z_{i-1}})$ for $1 \leq i\leq K$ with $S_{z_0}=S$.
 
If none of these explorations is successful, we declare that $Ep(F,S,K)$ {\it fails}, otherwise we say that it is {\it successful}. In the latter case,   there is a vertex $u$ with  $d(u,F) \leq \psi_1(K)$ together  with  two seed sets $F_{u,1}$ and $F_{u,2}$  of size $K$ (in fact, we even have three sets, but we will only use  two of them). 
\begin{lem} \label{lfw} We have
\begin{itemize}
\item[(i)] the number of vertices used in an experiment is  at most $\psi_3(K) =K \psi_2(K)$,
\item[(ii)]  there exists a positive constant $K_3 \geq K_2$, such that for all $K\geq K_3$, and  $n$ large enough
\begin{align*}
 \pp(Ep(F,S,K) \textrm{ is successful} \mid \kE(K), |S| \geq (1-\varkappa_2/2)n) \geq 2/3,
\end{align*}
with $\varkappa_2$  as in Lemma \ref{gea}. 
\end{itemize}
\end{lem}   
\begin{proof} Part (i) is immediate, since in an experiment, we perform $K$ explorations and each exploration uses at most $\psi_2(K)$ vertices. For (ii), we note that  by (i) and the assumption $|S| \geq (1-\varkappa_2/2)n$,  the source set $S_{z_i}$ has more than $(1- \varkappa_2)n$ vertices for all $i$.
  Hence, by Lemma \ref{lep} (ii),  for all $1 \leq i \leq  K $ and $n$ large enough
 \begin{align*}
 \pp(E_K(z_i,S_{z_{i-1}}) \textrm{ is successful} \mid \kE(K)) \geq \theta_1/2.
 \end{align*}
Thus on $\kE(K)$, the probability that the experiment fails is less than
 \begin{align*}
(1- \theta_1/2)^{K} \leq 1/3,
\end{align*}  
provided $K$ is large enough. 
\end{proof}
We  define {\bf Task II} as follows. First, we fix a constant $K\geq K_3$ and let 
$$\varepsilon_1 = \varkappa_2 / (3\psi_3(K)).$$
We label the seed sets {\it active} and make an order  as follows 
\begin{align*}
 F_{u_1,1} < F_{u_1,2} <F_{u_2,1} < \ldots < F_{u_{L_n},1} < F_{u_{L_n},2}.
\end{align*}
We perform $Ep(F_{u_{L_n},2}, U^*,K)$ and let $F_{u_{L_n,2}}$ be inactive.  If the experiment is successful, we find a vertex $u$ at distance  smaller than $ \psi_1(K)+1$ from $u_{L_n}$ and two seed sets $F_{u,1}$ and $F_{u,2}$  of size $K$.  We  now add $u$ in the sequence: $u_{L_{n+1}}=u$, label these sets $F_{u_{L_{n+1},1}}$ and $F_{u_{L_{n+1},2}}$ active, and make an order
$$F_{u_1,1}< \ldots < F_{u_{L_{n},1}} < F_{u_{L_{n+1},1}} < F_{u_{L_{n+1},2}}.$$
We then perform the  experiment of the newest active set i.e. the active set with the largest order. After an experiment of an active set, we let it be inactive and either get a new vertex with  two  active sets attached on it (if the experiment is successful), or get nothing (otherwise).

 Continue this procedure until  one of the three following conditions is  satisfied.  
\begin{itemize}
\item[$\bullet$] There is no more active set. We declare that Task II fails.
\item[$\bullet$] We do more than $[\varepsilon_1 n]$ experiments. We declare that Task II fails.
\item[$\bullet$] We have more than $[\varepsilon_1 n/4]$ active sets. We declare that Task II is successful.
\end{itemize}
\begin{prop} \label{pc2} For all $K\geq K_3$, we have
$$  \lim_{n \rightarrow \infty} \pp (\textrm{Task II is successful} \mid \kE(K),  \textrm{Task I is successful})=1.$$
\end{prop}
\begin{proof}
By Lemma \ref{lfw} (i),  the first $[\varepsilon_1 n]$ experiments use at most $[\varkappa_2 n/3]$ vertices. Therefore, during Task II, the source set  always has  cardinality larger than $(1-\varkappa_2/2)n$. Thus by Lemma \ref{lfw} (ii), on $\kE(K)$ during the time to perform this task, each experiment is successful with probability larger than $2/3$. 

Therefore,  the number of active sets  stochastically dominates a random walk $(R_i)$ satisfying $R_0 = 2L_n$ and 
\begin{align*}
R_{i+1} & = R_i +1 \textrm{ with prob. } 2/3 \\
R_{i+1} & = R_i -1 \textrm{ with prob. } 1/3. 
\end{align*} 
Define 
\begin{align*}
T_0 =\inf \{i: R_i =0\} \qquad \textrm{ and } \qquad T_1 = \inf \{i: R_i \geq [\varepsilon_1 n/4]\}.
\end{align*}
Then using the optional stopping  theorem, we get 
\begin{align*}
\pn(T_1 \leq T_0) = 1-o(1).
\end{align*}
On the other hand, the law of large numbers gives that
\begin{align*}
\pn(T_1 \leq [\varepsilon_1 n]) \geq \pn(R_{[\varepsilon_1n]} \geq [\varepsilon_1 n/4]) = 1-o(1).
\end{align*}
It follows from the last two inequalities that
\begin{eqnarray*}
 & &\pp(\textrm{Task II is successful} \mid \kE(K),  \textrm{Task I is successful}) \\
  &  \geq &  \pp(T_1 \leq \min \{T_0, [\varepsilon_1 n] \}) =1-o(1),
\end{eqnarray*}
which proves the result.
\end{proof}
\subsubsection{Proof of Proposition \ref{pcm}} By Propositions \ref{pc1}, \ref{pc2} and \eqref{kek}, we can assume that both Tasks I and II are successful. Then we have more than $[\varepsilon_1 n/4]$ active sets. Observe that a vertex is attached to at most  two active sets.  Then the number of vertices having at least one active set is larger than $[\varepsilon_1n/8]$. Therefore, w.h.p. $G_n$ belongs to the class $\kS([cn],2K)$ with $c=\varepsilon_1/8$. Moreover, $K$ can be chosen arbitrarily large, so   Proposition \ref{pcm} has been proved.   
\section{Contact process on Erdos-R\'enyi random graphs}
We recall the definition of $ER(n,p/n)$-the Erdos-R\'enyi graph with parameter $p$. Let $V_n=\{v_1,\ldots,v_n\}$ be the vertex set.  Then for $1 \leq i \neq j\leq n$, we independently draw an edge between $v_i$ and $v_j$ with probability $p/n$. 
\begin{prop} \label{ped}  Let $\tau_n $ be the extinction time of the contact process on $ER(n,p/n)$ starting with all  sites infected. There exists   a positive constant  $C$, such that for any $\lambda >0$ and  $ p\geq [C/h(\lambda)]!$,
 $$\pn(\tau_n \geq \exp(cn))=1 -o(1),$$
 with $h(\lambda)$ as in Proposition \ref{tng} and $c=c(\lambda)$ a positive constant.  Moreover, in this setting
\begin{align*} 
\frac{\tau_n}{\en (\tau_n)}\quad  \mathop{\longrightarrow}^{(\kD)}_{n\to \infty} \quad  \kE(1). 
\end{align*} 
 \end{prop}
Let us denote by $S(\ell,M)$ the graph obtained by gluing to each vertex in $\llbracket 1,n \rrbracket$ a star graph of size $M$. Similarly to Proposition \ref{tng}, we get that if $h(\lambda) M $ is large enough, then the extinction time of the contact process on $S(\ell,M)$ is exponential in $\ell \times M$ w.h.p. Therefore, similarly to Theorem \ref{mt}, Proposition \ref{ped} follows from the following lemma.
\begin{lem} \label{lscm}  For any $M$, there exists a positive constant $c$, such that   if $p \geq 16 M!$ then w.h.p. $ER(n,p/n)$ contains as a subgraph a copy of $S([cn],M)$.
\end{lem} 
To prove Lemma \ref{lscm}, we will use the following.
\begin{lem} \cite[Theorem 2]{AKS} \label{laks}
If $p >1$ then w.h.p.  $ER(n,p/n)$  contains a path of length  $\varkappa n$ with some positive constant $\varkappa = \varkappa (p)$. 
\end{lem}
\noindent{\it Proof of Lemma \ref{lscm}}. We first  define 
$$n_1=[n/2], \quad  \quad A=\{v_1, \ldots, v_{n_1-1}\} \quad \textrm{ and } \quad A^c=V_n \setminus A.$$
For $v_i \in A$ and $v_j \in A^c$, define
$$Y_{i,j}= 1(v_i \sim v_j).$$
 Then $(Y_{i,j})$ are i.i.d. Bernoulli  random variables with mean $p/n$. We set $B_0= \varnothing$ and
$$\sigma_1 = \inf \big \{ j \leq n : \sum \limits_{k=n_1}^j Y_{1,k} \geq M  \big \},$$
  with the  convention $\inf \varnothing = \infty.$
Define
\begin{displaymath}
B_1 = \left \{ \begin{array}{ll}
\{k:  k \leq \sigma_1 , Y_{1,k}=1 \} & \textrm{ if } \quad \sigma_1 < \infty ,\\
B_0 & \textrm{ if }\quad \sigma_1 = \infty .
\end{array} \right.
\end{displaymath} 
Suppose that $\sigma_i $ and $ B_i$  have been already defined. Then we set
$$\sigma_{i+1}= \inf \big \{ j \leq n : \sum \limits_{k=n_1}^j Y_{i+1,k}1(k \not \in B_i) \geq M \big \}, $$
and 
 \begin{displaymath}
B_{i+1} = \left \{ \begin{array}{ll}
B_i \cup \{k:  k \leq \sigma_{i+1} , k \not \in B_i,  Y_{i+1,k}=1 \} & \textrm{ if } \quad \sigma_{i+1} < \infty ,\\
B_i & \textrm{ if } \quad \sigma_{i+1} = \infty .
\end{array} \right.
\end{displaymath} 
Then  
\begin{displaymath}
|B_{i+1}| = \left \{ \begin{array}{ll}
|B_i|+M & \textrm{ if } \quad \sigma_{i+1} < \infty ,\\
|B_i| & \textrm{ if } \quad \sigma_{i+1} = \infty .
\end{array} \right.
\end{displaymath} 
Hence, for all $i \leq [n/M]$,
\begin{align} \label{bim}
|B_i| \leq i M.  
\end{align}
We now define
$$Y_i= 1(\sigma_i < \infty ).$$
Then 
\begin{eqnarray*}
\pn (Y_{i+1}=1 \mid B_i) &=& \pn \left( \sum \limits_{k=n_1}^n Y_{i+1,k} 1(k \not \in B_i) \geq M \mid B_i \right) \\
& = & \pn \left( \bin(n-n_1+1 -|B_i|, p/n) \geq M \mid B_i\right).
\end{eqnarray*}
It follows from  \eqref{bim} that when $i \leq [n/4M]$,  
$$n-n_1+1 - |B_i| \geq  [n/4].$$
  Thus
\begin{eqnarray} \label{bs}
\en (Y_{i+1} \mid B_i) &= & \pn (Y_{i+1} =1 \mid B_i) \geq  \pn(\bin([n/4],p/n) \geq M) \notag \\
 &\geq & \pn(\poi(p/4) \geq M)/2 \geq  1/(M-1)!
\end{eqnarray}
with $p \geq 8$ and $n$ large enough. Now we set 
$$\Gamma = \{i \leq [n/4M]: Y_i =1\},$$
and let 
\begin{equation} \label{h2}
Z_k= \sum \limits_{i=1}^k (Y_i - \en(Y_i \mid B_{i-1})).
\end{equation} 
Then $(Z_k)$ is a $(\sigma(B_k))$- martingale satisfying $|Z_k-Z_{k-1}| \leq 1$ for all $1 \leq k\leq [n/4M]$. Thus it follows from   Doob's martingale inequality  that
$$\pn \left( \max \limits_{k \leq \ell} |Z_k| \geq x \right) \leq \frac{\en(Z_{\ell}^2)}{x^2} \leq \frac{\ell}{x^2} \quad \textrm{ for all } \quad \ell \leq [n/4M], \,\, x >0.$$
In particular,  w.h.p.
\begin{equation} \label{znm}
|Z_{[n/4M]}| = o( n).
\end{equation}
It follows from $(\ref{bs}), \eqref{h2}$ and $ (\ref{znm})$  that w.h.p.
\begin{eqnarray*}
|\Gamma| = \sum_{i=1}^{[n/4M]} Y_i  &=&  Z_{[n/4M]} + \sum_{i=1}^{[n/4M]}  \en(Y_i \mid B_{i-1}) \notag\\
&\geq & Z_{[n/4M]} + [n/4M]/(M-1)!  \geq   n/(8M!).
\end{eqnarray*}
 Conditionally on $|\Gamma| \geq n /(8M!)$, the graph induced in $\Gamma$ contains a Erdos-R\'enyi graph, say  $H_n$, of size $[n /(8M!)]$ with probability of connection $p/n$. Observe that  $H_n$ is super critical  when
$$p /(8M!) \geq 2. $$
 Therefore, by Lemma \ref{laks},  w.h.p. $H_n$  contains a path of length $\varkappa n$, with $\varkappa = \varkappa(p,M)$ if $p\geq 16M!$.

On the other hand,   each vertex in $\Gamma$ has at least $M$  disjoint neighbors in $A^c$. 
Combining these facts proves  the result. \hfill $\square$

\begin{rem} \emph{We can use Proposition \ref{ped} to prove the lower bound on the extinction time in Theorem \ref{mt} for a small class of weights. Indeed, for any $p$ we define
$$A_p=\{v_i: w_i \geq \sqrt{4p \en(w)}\}.$$
Then 
$$|A_p| = (\gamma_p +o(1)) n,$$
with $$\gamma_p= \pp(w \geq \sqrt{4p \en(w)}).$$
On the other hand, for any $v_i $ and $ v_j$ in $A_p$,
\begin{eqnarray*}
\pp(v_i \sim v_j) &= &1 -\exp(-w_i w_j / \ell_n) \\
& \geq &\frac{w_i w_j}{2 \ell_n} 1(w_i w_j \leq \ell_n) + \frac{1}{2} 1(w_i w_j > \ell_n) \\
& \geq &\frac{4p \en(w)}{2\ell_n}  \geq \frac{p}{n}, 
\end{eqnarray*}
since $\ell_n = n (\en(w)+o(1))$.  Therefore, the graph induced in $A_p$ contains a Erdos-R\'enyi graph $ER([\gamma_p n], p/n)$. By proposition \ref{ped}, if $p \gamma_p \rightarrow \infty$,  for all $\lambda >0$, w.h.p.  the extinction time of the contact process on $ER([\gamma_p n],p/n)$ is exponential in $n$. Hence, when
\begin{align} \label{cpw}
\lim_{p \rightarrow \infty} p \pp(w \geq \sqrt{4p \en(w)}) = \infty, 
\end{align}
 w.h.p. the extinction time on the IRG($w$) is exponential in $n$ for all $\lambda >0$. \\
The condition \eqref{cpw} is satisfied for some weight $w$, for example the power-law distribution with exponent between $2$ and $3$.
} 
\end{rem}

\vspace{0.2 cm}
\noindent {\bf Acknowledgments}. I am  grateful to  Bruno Schapira for his help and  suggestions during the preparation of this work.  I wish to thank also the anonymous referee for carefully reading the manuscript and many valuable comments. This work is supported by the Vietnam National Foundation for Science and Technology Development (NAFOSTED) under  Grant number 101.03--2017.07.

\end{document}